\documentclass[letterpaper]{amsart}

\usepackage[utf8]{inputenc}
\usepackage[T1]{fontenc}
\usepackage{lmodern}
\usepackage{amssymb}
\usepackage{mathtools}
\usepackage{tikz-cd}
\usepackage[pdfusetitle]{hyperref}
\def\arxiv{\href{http://arxiv.org/abs/}{\texttt{arXiv:}}}
\def\doi{\href{http://dx.doi.org/}{doi:}}

\theoremstyle{plain}
\newtheorem{theorem}{Theorem}[section]
\newtheorem{proposition}[theorem]{Proposition}
\newtheorem{lemma}[theorem]{Lemma}
\newtheorem{corollary}[theorem]{Corollary}

\theoremstyle{definition}
\newtheorem{example}[theorem]{Example}
\newtheorem{question}[theorem]{Question}

\theoremstyle{remark}
\newtheorem*{acknowledgements}{Acknowledgements}

\numberwithin{equation}{section}

\def\cf{\emph{cf.}}

\let\newterm\relax

\def\N{\mathbb N}
\def\Z{\mathbb Z}

\def\R{\mathbb R}
\def\C{\mathbb C}
\def\kk{\Bbbk}
\def\RP{\mathbb{RP}}
\def\CP{\mathbb{CP}}
\def\GG{\mathcal{G}}
\def\XX{\mathcal{X}}
\def\YY{\mathcal{Y}}
\def\EE{\mathcal{E}}
\def\BB{\mathcal{B}}
\def\SS{\mathcal{S}}

\def\SP#1#2{SP^{#2}(#1)}
\def\GP#1#2{#1^{#2}}

\def\timesG{\mathop{\times_{G}}}
\def\timesGG{\mathop{\times_{\GG}}}
\def\otimesGG{\mathop{\otimes_{N(\GG)}}}
\DeclareMathOperator{\Tor}{Tor}

\def\triv{_{\mathrm{triv}}}
\begin{document}

\title[Symmetric products of equivariantly formal spaces]{Symmetric products of\\equivariantly formal spaces}
\author{Matthias Franz}
\thanks{The author was supported by an NSERC Discovery Grant.}
\address{Department of Mathematics, University of Western Ontario,
      London, Ont.\ N6A\;5B7, Canada}
\email{mfranz@uwo.ca}
      
% 14P25 Topology of real algebraic varieties
% 55N91 Equivariant homology and cohomology
% 55S15 Symmetric products, cyclic products

\begin{abstract}
  Let \(X\) be a CW~complex with a continuous action of a topological group~\(G\).
  We show that if \(X\) is equivariantly formal for singular cohomology
  with coefficients in some field~\(\kk\), then so are all symmetric products of~\(X\)
  and in fact all its \(\Gamma\)-products.
  In particular, symmetric products
  of quasi-projective M-varieties are again M-varieties.
  This generalizes a result by Biswas and D'Mello
  about symmetric products of M-curves.
  We also discuss several related questions.
\end{abstract}

\maketitle

\section{Statement of the results}

Let \(X\) be a complex algebraic variety
with an anti-holomorphic involution~\(\tau\).
Then the sum of the \(\Z_{2}\)-Betti numbers of the fixed point set~\(X^{\tau}\)
cannot exceed the corresponding sum for~\(X\). In case of equality,
\begin{equation}
  \label{eq:maximal}
  \dim H^{*}(X^{\tau};\Z_{2}) = \dim H^{*}(X;\Z_{2}),
\end{equation}
one calls \(X\) \newterm{maximal} or an \newterm{M-variety}.
Maximal varieties are an important object of study in real algebraic geometry.

Let \(n\ge0\).
The \(n\)-th symmetric product~\(\SP{X}{n}\) of~\(X\) is the quotient
of the Cartesian product~\(X^{n}\)
by the canonical action of the symmetric group~\(S_{n}\);
\(\SP{X}{0}\) is a point.
If \(X\) is quasi-projective,
then \(\SP{X}{n}\)
is again a complex algebraic variety
equipped with an anti-holomorphic involution induced by~\(\tau\).

Assume that \(X\) is a compact connected Riemann surface of genus~\(g\ge0\).
In this case, Biswas--D'Mello~\cite{BiswasDMello}
have recently shown that if \(X\) is maximal,
then so is \(\SP{X}{n}\) for~\(n\le 3\) and~\(n\ge 2g-1\).
The main purpose of the present note is to point out that this conclusion holds in far greater generality.

A continuous involution~\(\tau\) on a topological space~\(X\) is the same as
a continuous action of the group~\(C=\{1,\tau\}\cong\Z_{2}\).
For many ``nice'' \(C\)-spaces,
including the algebraic varieties considered above,
the equality~\eqref{eq:maximal} is equivalent to the surjectivity
of the canonical restriction map \(H_{C}^{*}(X;\Z_{2})\to H^{*}(X;\Z_{2})\) from equivariant
to ordinary cohomology, see Proposition\nobreakspace \ref {thm:betti-sum-equiv-cohom} below.

We recall the definition of (Borel) equivariant cohomology,
\cf~\cite[Sec.~III.1]{tomDieck:1987}.
Let \(G\) be a topological group, and let \(EG\to BG\) be the universal \(G\)-bundle;
for \(G=C\) this is the bundle~\(S^{\infty}\to\RP^{\infty}\).
The equivariant cohomology of a \(G\)-space~\(X\) with coefficients in the field~\(\kk\)
is defined as \(H_{G}^{*}(X;\kk)=H^{*}(X_{G};\kk)\),
where \(H^{*}(-)\) denotes singular cohomology
and the Borel construction~\(X_{G}=EG\timesG X\) is the
quotient of~\(EG\times X\) by the diagonal \(G\)-action.

If the inclusion of the fibre~\(X\hookrightarrow X_{G}\)
induces a surjection in cohomology,
then \(X\) is called \newterm{equivariantly formal} over~\(\kk\).
This condition is equivalent to the freeness of~\(H_{G}^{*}(X;\kk)\) over~\(H^{*}(BG;\kk)\)
if \(G\) is for instance a compact connected Lie group or a connected complex algebraic group,
see Proposition\nobreakspace \ref {thm:equiv-formal-freeness}.
Many spaces are known to be equivariantly formal over~\(\R\),
for example compact Hamiltonian \(G\)-manifolds
for \(G\)~a compact connected Lie group
\cite{Frankel:1959},~\cite[Prop.~5.8]{Kirwan:1984},
or rationally smooth compact complex algebraic \(G\)-varieties
for \(G\)~a reductive connected algebraic group
\cite[Thm.~14.1]{GoreskyKottwitzMacPherson:1998},~\cite{Weber:2003}.

Let \(\Gamma\subset S_{n}\) be a subgroup.
The \(\Gamma\)-product~\(\GP{X}{\Gamma}\) of a topological space~\(X\) is the quotient of~\(X^{n}\)
by the canonical action of~\(\Gamma\), \cf~\cite[Def.~7.1]{Dold:1957}.\footnote{%
  Despite the similar notation, the \(\Gamma\)-product~\(\GP{X}{\Gamma}\)
  should not be confused with the fixed point set~\(X^{G}\)
  of the \(G\)-action on~\(X\).}
For~\(\Gamma=1\) one obtains the Cartesian product~\(X^{n}\)
and for~\(\Gamma=S_{n}\) the \(n\)-th symmetric product of~\(X\)
considered above for quasi-projective varieties.
Note that any continuous \(G\)-action on~\(X\) induces one on~\(\GP{X}{\Gamma}\).

Our generalization of Biswas--D'Mello's result now reads as follows:

\begin{theorem}
  \label{thm:symprodformal}
  Let \(G\) be a topological group and \(X\) a CW~complex with a continuous \(G\)-action.
  Let \(\kk\) be a field and \(\Gamma\) a subgroup of~\(S_{n}\) for some~\(n\ge0\).
  If \(X\) is equivariantly formal over~\(\kk\), then so is \(\GP{X}{\Gamma}\).
\end{theorem}

\begin{corollary}
  \label{thm:symprodmvar}
  If \(X\) is a quasi-projective M-variety,
  then \(\GP{X}{\Gamma}\) is an M-va\-ri\-ety.
  In particular, symmetric products of quasi-projective M-varieties
  are again M-varieties.
\end{corollary}

We have the following partial converse to Theorem\nobreakspace \ref {thm:symprodformal}.

\begin{proposition}
  \label{thm:converseGP}
  With the same notation as before,
  assume \(n\ge1\) and that \(X\) has fixed points.
  If \(\GP{X}{\Gamma}\) is equivariantly formal over~\(\kk\), then so is \(X\).
\end{proposition}

\begin{example}
  Let \(X=\CP^{1}\). Then \(\SP{X}{n}\) is homeomorphic to~\(\CP^{n}\),
  as can be seen by identifying \(\C^{n+1}\) with complex binary forms
  of degree~\(n\) and invoking the fundamental theorem of algebra.

  Let \(\tau\) be complex conjugation on~\(\CP^{1}\).
  This is an anti-holomorphic involution with fixed point set~\(\RP^{1}\);
  \(\CP^{1}\) therefore is maximal.
  The induced involution on~\(\CP^{n}\) is again complex conjugation.
  Hence \(\SP{X}{n}^{\tau}=\RP^{n}\), showing that \(\SP{X}{n}\) is maximal, too.

  Now consider the holomorphic involution on~\(X\) given in homogeneous coordinates
  by~\([x_{0}:x_{1}]\mapsto[x_{0}:-x_{1}]\). This action is also equivariantly
  formal over~\(\Z_{2}\) with the two fixed points~\([1:0]\) and~\([0:1]\).
  The induced involution on~\(\CP^{n}\) is
  \begin{equation}
    [x_{0}:\dots:x_{n}] \mapsto [x_{0}:-x_{1}:x_{2}:\dots:(-1)^{n}\,x_{n}].
  \end{equation}
  For~\(n\ge1\),
  its fixed point set is the disjoint union of~\(\CP^{k}\) and~\(\CP^{l}\),
  where \(k\) and~\(l\) are obtained by rounding \((n-1)/2\) respectively up and down
  to the next integer. The Betti sum of the fixed point set
  is \((k+1)+(l+1)=n+1\). This is the same as for~\(\CP^{n}\),
  which again confirms Theorem\nobreakspace \ref {thm:symprodformal}.

  Finally, let \(\tau\) be the anti-holomorphic involution given by~\([x_{0}:x_{1}]\mapsto[\bar x_{1}:-\bar x_{0}]\).
  It is fixed-point free and corresponds to the antipodal map on~\(S^{2}\).
  The fixed point set of~\(\SP{X}{2}=\CP^{2}\) is homeomorphic to
  the orbit space~\(X/C\approx\RP^{2}\). Hence \(\SP{X}{2}\) is maximal, while \(X\) itself is not.
  This illustrates that Proposition\nobreakspace \ref {thm:converseGP} may fail for actions without fixed point.
\end{example}

A \(G\)-space~\(X\) is equivariantly formal if and only if
\(G\) acts trivially on~\(H^{*}(X;\kk)\) and 
the Serre spectral sequence for the bundle~\(X\to X_{G}\to BG\) with coefficients in~\(\kk\)
degenerates at the second page
\cite[Prop.~III.1.17]{tomDieck:1987},~\cite[\S VI.5.5]{Lamotke:1968}.
One can study these two conditions separately.

\begin{proposition}
  \label{thm:trivactioncohom}
  Assume the same notation as in Theorem\nobreakspace \ref {thm:symprodformal}.
  If \(G\) acts trivially on~\(H^{*}(X)\), then so it does on~\(H^{*}(\GP{X}{\Gamma})\).
  The converse holds if \(X^{G}\ne\emptyset\) and \(n\ge1\).
\end{proposition}

The degeneration of the Serre spectral sequence is a more delicate matter.
Consider again a compact connected Riemann surface~\(X\)
with an anti-holo\-mor\-phic involution~\(\tau\),
and assume that it has fixed points. In this context,
Baird~\cite[Prop.~3.9]{Baird} has recently shown that
the Serre spectral sequence with coefficients in~\(\Z_{2}\)
for the Borel construction of any symmetric product of~\(X\)
degenerates at the second page.
It is not difficult to extend this to compact Riemann surfaces
that are not connected or without fixed points.
Recall that \(C=\{1,\tau\}\).

\begin{proposition}
  \label{thm:SPm-weakly-tight}
  Let \(X\) be a compact, not necessarily connected Riemann surface
  with an anti-holomorphic involution~\(\tau\).
  If the Serre spectral sequence for~\(X_{C}\)
  with coefficients in~\(\Z_{2}\)
  degenerates at the second page,
  then so does the one for~\(\SP{X}{n}_{C}\).
\end{proposition}

\begin{question}[Baird]
  Let the notation be as in Theorem\nobreakspace \ref {thm:symprodformal}.
  If the Serre spectral sequence for~\(X_{G}\) with coefficients in~\(\kk\) degenerates at the second page,
  does the same hold true for the Borel construction of~\(\GP{X}{\Gamma}\)\,?
\end{question}

In a different direction, the notion of equivariant formality
has been extended to that of a syzygy in equivariant cohomology
by Allday--Franz--Puppe~\cite{AlldayFranzPuppe:orbits1} (\(G\)~a torus)
and Franz~\cite{Franz:nonab} (\(G\) a compact connected Lie group).
Let \(r\) be the rank of such a~\(G\),
so that \(H^{*}(BG;\R)\) is a polynomial algebra in \(r\)~variables of even degrees.
For~\(1\le k\le r\),
the \(k\)-th syzygies over~\(H^{*}(BG;\R)\) interpolate between torsion-free modules
(\(k=1\)) and free ones (\(k=r\)).
Since a \(G\)-space~\(X\) is equivariantly formal over~\(\R\)
if and only if \(H_{G}^{*}(X;\R)\) is a free module over~\(H^{*}(BG;\R)\),
Theorem\nobreakspace \ref {thm:symprodformal} can be restated as follows: If \(H_{G}^{*}(X;\R)\)
is an \(r\)-th syzygy over~\(H^{*}(BG;\R)\), then so is \(H_{G}^{*}(\GP{X}{\Gamma};\R)\).
The following example, whose details appear at the end of the paper, shows that this result does not extend to smaller syzygy orders.

\begin{example}
  \label{ex:SP2-SigmaT}
  Let \(G=(S^{1})^{r}\) be a torus;
  \(H^{*}(BG;\R)\) is a polynomial algebra in \(r\)~indeterminates of degree~\(2\) over~\(\R\).
  Let \(X=\Sigma G\), the suspension of the torus.
  Then \(H_{G}^{*}(X;\R)\) is given by pairs of polynomials with the same constant term.
  Hence, \(H_{G}^{*}(X;\R)\) is torsion-free, but not free for~\(r\ge2\) (not even a second syzygy),
  see~\cite[Example~3.3]{Allday:2008}.

  There is a \(G\)-stable filtration of~\(\SP{X}{2}\) of length~\(2\)
  whose associated spectral sequence converging to~\(H_{G}^{*}(\SP{X}{2};\R)\)
  has the the property
  that second column of the limit page is finite-dimensional over~\(\R\) and non-zero for~\(r\ge3\).
  Since this column
  is an \(H^{*}(BG;\R)\)-submodule of~\(H_{G}^{*}(\SP{X}{2};\R)\),
  it follows that \(H_{G}^{*}(\SP{X}{2};\R)\) has torsion for~\(r\ge3\).
\end{example}

\begin{acknowledgements}
  I thank Tom Baird and Volker Puppe for stimulating discussions.
  I am particularly indebted to Volker Puppe for drawing my attention
  to Dold's work~\cite{Dold:1957}
  and for suggesting to look at the \(G\)-action in the cohomology of~\(X\).
\end{acknowledgements}

\section{Proofs}

We are going to show that Theorem\nobreakspace \ref {thm:symprodformal} is a consequence
of Dold's results~\cite{Dold:1957} about the homology of symmetric products.
For the sake of completeness, let us first justify the claims made previously
regarding the Betti sum of the fixed point set
and equivariant formality.
In this section, all (co)homology is taken with coefficients in a field~\(\kk\).

\begin{proposition}
  \label{thm:betti-sum-equiv-cohom}
  Let \(p\) be a prime and \(r\in\N\). Let \(\kk=\Z_{p}\), \(G=(\Z_{p})^{r}\) and
  \(X\) a smooth \(G\)-manifold or real analytic \(G\)-variety
  with finite Betti sum. Then
  \begin{equation*}
    \dim H^{*}(X^{G}) \le \dim H^{*}(X)
  \end{equation*}
  with equality
  if and only if \(X\) is equivariantly formal.
\end{proposition}

\begin{proof}
  This holds in fact for a much larger class of \(G\)-spaces~\(X\) including
  finite-dimensional \(G\)-CW~complexes with finite Betti sum,
  \cf~\cite[Prop.~III.4.16]{tomDieck:1987}.
  It thus suffices to observe that smooth \(G\)-manifolds
  and real analytic (even subanalytic) \(G\)-varieties
  are \(G\)-CW complexes,
  see \cite[Thm., p.~199]{Illman:1978}
  and~\cite[Cor.~11.6]{Illman:2000}.
\end{proof}

\begin{proposition}
  \label{thm:equiv-formal-freeness}
  Let \(G\) be a connected group with homology of finite type
  and let \(X\) be a \(G\)-space.
  Then \(X\) is equivariantly formal over~\(\kk\) if and only if \(H_{G}^{*}(X)\)
  is a free \(H^{*}(BG)\)-module.
\end{proposition}

\begin{proof}
  Recall that \(BG\) is simply connected and with homology of finite type
  if (and only if) \(G\) is connected and with homology of finite type,
  \cf~\cite[Cor.~7.29]{McCleary:2001} and the proof of Lemma\nobreakspace \ref {thm:HGG-HG} below.
  
  Assume that \(X\) is equivariantly formal over~\(\kk\), so that
  the restriction to the fibre~\(H_{G}^{*}(X)=H^{*}(X_{G})\to H^{*}(X)\) is surjective.
  The Leray--Hirsch theorem~\cite[VI.8.2]{Lamotke:1968},~\cite[Prop.~III.1.18]{tomDieck:1987}
  then implies that \(H^{*}(X_{G})\) is a free module over~\(H^{*}(BG)\).

  For the converse we use the Eilenberg--Moore spectral sequence~%
  \cite[Thm.~3.6]{Smith:1967}, \cite[Cor.~7.16]{McCleary:2001}
  \begin{equation}
    E_{2}^{p,*} = \Tor_{H^{*}(BG)}^{p}(H^{*}(X_{G}),\kk)
    \;\Rightarrow\; H^{*}(X).
  \end{equation}
  Since \(H^{*}(X_{G})\) is free over~\(H^{*}(BG)\),
  the higher \(\Tor\)-modules vanish and the spectral sequence degenerates
  at the second page, whence
  \begin{equation}
    H^{*}(X) \cong H^{*}(X_{G})\otimes_{H^{*}(BG)}\kk.
  \end{equation}
  It follows that the restriction to the fibre is surjective
  because the canonical map
  \begin{equation}
    H^{*}(X_{G})\otimes_{H^{*}(BG)}\kk \to H^{*}(X)
  \end{equation}
  induced by the restriction map
  corresponds to the edge homomorphism of the
  spectral sequence,
  see~\cite[Prop.~1.4\('\)]{Smith:1967}.
\end{proof}

We now turn to the proof of Theorem\nobreakspace \ref {thm:symprodformal}.
As in~\cite{Dold:1957}, it will be convenient to work with simplicial sets \cite{Lamotke:1968},~\cite{May:1968}.
We write \(\SS(X)\) for the simplicial set of singular simplices in a topological space~\(X\),
and \(H(\XX)\) for the homology of a simplicial set~\(\XX\) with coefficients in~\(\kk\)
as well as \(H^{*}(\XX)\) for its cohomology.

Recall that the singular simplices in a topological group~\(G\) form a simplicial group~\(\SS(G)\),
and those in a \(G\)-space~\(X\) a simplicial \(\SS(G)\)-set~\(\SS(X)\).
For any simplicial group~\(\GG\) there is a canonical universal \(\GG\)-bundle~\(E\GG\to B\GG\) \cite[\S 21]{May:1968}, and
for any simplicial \(\GG\)-set~\(\XX\) one can define
its equivariant cohomology~\(H_{\GG}^{*}(\XX)=H^{*}(E\GG\timesGG\XX)\).

The following observation is presumably not new, but we were unable to locate a suitable reference.

\begin{lemma}
  \label{thm:HGG-HG}
  Let \(G\) be a topological group and \(X\) a \(G\)-space. Then there is an isomorphism of graded \(\kk\)-algebras
  \begin{equation*}
    H_{\SS(G)}^{*}(\SS(X)) \to H_{G}^{*}(X),
  \end{equation*}
  compatible with the restriction maps to~\(H^{*}(\SS(X))=H^{*}(X)\).
\end{lemma}

\begin{proof}
  Let \(\GG\) be a simplicial group, \(\EE\to\BB\) a principal \(\GG\)-bundle
  and \(\XX\) a simplicial \(\GG\)-set.
  By a theorem of Moore's~\cite[Thm.~7.28]{McCleary:2001},
  there is a spectral sequence which is natural in~\((\EE,\GG,\XX)\)
  and converging to~\(H(\EE\timesGG\XX)\) with second page
  \begin{equation}
    \label{eq:Tor}
    \Tor^{H(\GG)}(H(\EE),H(\XX)).
  \end{equation}

  This can be seen as follows:
  Denote the normalized chain functor with coefficients in~\(\kk\) by~\(N(-)\).
  It is a consequence of the twisted Eilenberg--Zilber theorem that
  the complexes~\(N(\EE\timesGG\XX)\) and \(N(\EE)\otimesGG N(\XX)\) are homotopic
  \cite[Prop.~4.6\(_{*}\)]{Gugenheim:1972}. The latter complex is homotopic
  to the bar construction~\(B(N(\EE),N(\GG),N(\XX))\) because
  the differential \(N(\GG)\)-modules~\(N(\XX)\) and \(B(N(\GG),N(\GG),N(\XX))\) are homotopic,
  \cf~the proof of~\cite[Prop.~7.8]{McCleary:2001}.
  The former bar construction can be filtered in such a way that the second page
  of the associated spectral sequence equals \eqref{eq:Tor}.

  Now set \(\GG=\SS(G)\) and \(\XX=\SS(X)\).
  We have \(\SS(EG\timesG X)=\SS(EG)\timesGG\XX\), and
  by~\cite[Lemma~21.9]{May:1968} there is a \(\GG\)-map~\(\SS(EG)\to E\GG\).
  So we only have to show that the induced map
  \begin{equation}
    \label{eq:SXG-SXGG}
    \SS(EG)\timesGG\XX \to E\GG\timesGG\XX
  \end{equation}
  induces an isomorphism in cohomology.
  But this follows from Moore's theorem because
  the map between the second pages of the spectral sequences
  \begin{equation}
    \Tor^{H(\GG)}(H(EG),H(\XX)) \to \Tor^{H(\GG)}(H(E\GG),H(\XX))
  \end{equation}
  is an isomorphism
  as both~\(EG\) and~\(E\GG\) are contractible.
  Hence \eqref{eq:SXG-SXGG} induces an isomorphism both in homology and cohomology.
\end{proof}

With a similar spectral sequence argument,
one can show that if a map~\(f\colon\XX\to\YY\) of simplicial \(\GG\)-sets
induces an isomorphism in homology, then it also does so in equivariant cohomology.

\begin{lemma}
  \label{thm:symprod-inj}
  Let \(f\colon\XX\to\YY \) be a map of simplicial sets.
  If \(H(f)\) is injective, then so is
  \(H(\GP{f}{\Gamma})\colon H(\GP{\XX}{\Gamma})\to H(\GP{\YY}{\Gamma})\).
\end{lemma}

\begin{proof}
  Recall that the \(\Gamma\)-product is a functor on the category of simplicial vector spaces \cite[\S 6.2]{Dold:1957}.
  We write the simplicial \(\kk\)-vector spaces
  of chains in~\(\XX\) and~\(\YY\)
  as~\(C(\XX)\) and~\(C(\YY)\), respectively.
  
  If \(H(f)\) is injective, it admits a retraction~\(R\colon H(\YY)\to H(\XX)\).
  By~\cite[Prop.~3.5]{Dold:1957}, there is a morphism of
  simplicial vector spaces~\(r\colon C(\YY)\to C(\XX)\)
  such that \(H(r)=R\).
  By functoriality, \(H(\GP{r}{\Gamma})\) then is a retraction of \(H(\GP{f}{\Gamma})\).
\end{proof}

\begin{proof}[Proof of Theorem\nobreakspace \ref {thm:symprodformal}]
  Let \(\GG=\SS(G)\).
  We first consider a simplicial \(\GG\)-space~\(\XX\).
  Let \(e_{0}\in (E\GG)_{0}\) be a fixed base point.
  By abuse of notation, we use the same symbol for any degeneration of~\(e_{0}\).
  Elements of~\(E\GG \timesGG\XX\) are written in the form~\([e,x]\) with~\(e\in E\GG\) and \(x\in\XX\).

  Noting that \(\GP{\XX}{\Gamma}\) is again a \(\GG\)-space, we consider the commutative diagram
  \begin{equation}
  \begin{tikzcd}
    \GP{\XX}{\Gamma} \arrow{dr}[pos=0.3,below=1ex]{\gamma} \arrow{rr}{\alpha} & & E\GG \timesGG \GP{\XX}{\Gamma} \arrow{dl}[pos=0.4]{\beta} \\
    & \;\;\;\GP{(E\GG \timesGG \XX)}{\Gamma} \rlap{,}
  \end{tikzcd}
  \end{equation}
  where \(\alpha\) is the inclusion of the fibre~\(\GP{\XX}{\Gamma}\),
  \begin{align}
    \alpha\colon [x_{1},\dots,x_{n}] &\mapsto \bigl[ e_{0}, [x_{1},\dots,x_{n}] \bigr], \\
  \shortintertext{\(\beta\) is the map}
    \beta\colon \bigl[ e, [x_{1},\dots,x_{n}] \bigr] &\mapsto \bigl[ [e,x_{1}],\dots,[e,x_{n}] \bigr]
  \shortintertext{and \(\gamma\) is the \(\Gamma\)-product of the inclusion of the fibre~\(\iota\colon\XX\hookrightarrow E\GG \timesGG \XX\),}
    \gamma\colon [x_{1},\dots,x_{n}] &\mapsto \bigl[ [e_{0},x_{1}],\dots,[e_{0},x_{n}] \bigr].
  \end{align}
  By assumption, \(H^{*}(\iota)\) is surjective. Equivalently, \(H(\iota)\) is injective.
  By Lemma~\ref{thm:symprod-inj}, this implies that \(H(\gamma)=H(\beta)H(\alpha)\)
  is injective and therefore also \(H(\alpha)\). Hence \(H^{*}(\alpha)\) is surjective.
  This proves the simplicial analogue of our claim.

  To deduce the topological result from this,
  consider the canonical maps
  \begin{equation}
    \bigl|\GP{\SS(X)}{\Gamma}\bigr| \to \GP{|\SS(X)|}{\Gamma} \to \GP{X}{\Gamma},
  \end{equation}
  where \(|-|\) denotes topological realization.
  As explained in the proof of~\cite[Thm.~7.2]{Dold:1957},
  the first map is a homeomorphism between compact subsets, and the second one
  is a homotopy equivalence because \(|\SS(X)|\to X\) is so
  for the CW~complex~\(X\).
  As a consequence, the
  \(\GG\)-equivariant map~\(\SS(X)^{\Gamma}\to \SS(X^{\Gamma})\) is a quasi-isomorphism.
  Hence the surjectivity of the top row
  in the diagram
  \begin{equation}
  \begin{tikzcd}
    H_{\GG}^{*}(\GP{\SS(X)}{\Gamma}) \arrow{d}[left]{\cong} \arrow{r} & H^{*}(\GP{\SS(X)}{\Gamma}) \arrow{d}{\cong} \\
    H_{\GG}^{*}(\SS(\GP{X}{\Gamma})) \arrow{r} & H^{*}(\SS(\GP{X}{\Gamma}))
  \end{tikzcd}
  \end{equation}
  implies that of the bottom row.
  We conclude the proof with Lemma~\ref{thm:HGG-HG}.
\end{proof}

\begin{proof}[Proof of Corollary\nobreakspace \ref {thm:symprodmvar}]
  Recall that algebraic varieties are finite-dimensional CW complexes
  (see the proof of Proposition~\ref{thm:betti-sum-equiv-cohom}) with finite Betti sum.
  Because \(X\) is quasi-projective, \(\GP{X}{\Gamma}\) is again an algebraic variety,
  \cf~\cite[Example~6.1]{Dolgachev:2003}. By what we have said in the introduction,
  it is enough to verify that \(\GP{X}{\Gamma}\) is equivariantly formal with
  respect to complex conjugation. This follows from Theorem\nobreakspace \ref {thm:symprodformal}.
\end{proof}

\begin{lemma}
  \label{thm:surjection}
  Let \(X\) be a CW~complex, \(y\in X\) and \(\Gamma\subset S_{n}\) for some~\(n\ge1\).
  Then the map
  \begin{equation*}
    f\colon X\to\GP{X}{\Gamma},
    \quad
    x\mapsto[x,y,\dots,y]
  \end{equation*}
  induces an injection in homology.
\end{lemma}

\begin{proof}
  Let \(V\) and~\(W\) be \(\kk\)-vector spaces. Following~\cite[\S 8]{Dold:1957},
  we write \((V,W)^{1}\) for the \(\Gamma\)-submodule
  \begin{equation}
    \bigoplus_{k=1}^{n}\, W^{\otimes(k-1)}\otimes V\otimes W^{\otimes(n-k)}
    \subset (V\oplus W)^{\otimes n},
  \end{equation}
  which is in fact a direct summand.
  Moreover, \((V,W)^{1}/\Gamma\) is isomorphic to the direct sum of \(b\)~copies of~\(V\otimes W^{n-1}\),
  where \(b\) is the number of \(\Gamma\)-orbits in~\(\{1,\dots,n\}\).
  
  As in~\cite[\S 8]{Dold:1957}, this construction carries over to simplicial vector spaces.
  The choice of a base point~\(y\in X\) determines a splitting~\(C(X)=\tilde C(X)\oplus\kk\),
  \cf~\cite[\S 9]{Dold:1957}. Moreover, \(H(\GP{X}{\Gamma})\) contains \(b\)~summands
  of the form~\(\tilde H(X)\otimes\kk^{\otimes(n-1)}\),
  and \(H(f)\) maps \(\tilde H(X)\) isomorphically onto one of them.
  Hence \(H(X)\) injects into~\(H(\GP{X}{\Gamma})\).
\end{proof}

\begin{proof}[Proof of Proposition\nobreakspace \ref {thm:converseGP}]
  By Lemma\nobreakspace \ref {thm:surjection}, any choice of base point~\(y\in X\)
  gives a map~\(f\) inducing a surjection in cohomology.
  If \(y\) is a fixed point, then \(f\) is equivariant.
  The left and bottom arrow in the commutative diagram
  \begin{equation}
    \begin{tikzcd}
      H^{*}(X) & H_{G}^{*}(X) \arrow{l} \\
      H^{*}(\GP{X}{\Gamma}) \arrow{u}{H^{*}(f)} & H_{G}^{*}(\GP{X}{\Gamma}) \arrow{l} \arrow{u}[right]{H_{G}^{*}(f)}
    \end{tikzcd}
  \end{equation}
   are surjective, hence so is the top one.
\end{proof}

\begin{proof}[Proof of Proposition\nobreakspace \ref {thm:trivactioncohom}]
  As in the proof of Theorem\nobreakspace \ref {thm:symprodformal},
  we can consider the \(\Gamma\)-product of the simplicial vector space~\(C(X)\) instead of~\(X^{\Gamma}\)
  since \(X\) is a CW~complex.

  Let \(g\in G\); it induces a map~\(a_{g}\colon X\to X\), \(x\mapsto gx\).
  By assumption, \(H(a_{g})\) is the identity. Hence \(N(a_{g})\colon N(X)\to N(X)\)
  is homotopy equivalent to the identity map on the normalized chain complex of~\(X\), \cf~\cite[Prop.~II.4.3]{Dold:1980}.
  By~\cite[Cor.~2.7]{Dold:1957} this implies that the morphism~ of simplicial vector spaces~\(C(a_{g})\colon C(X)\to C(X)\)
  is also homotopy equivalent to the identity.
  Because \(\Gamma\)-products preserve homotopy \cite[Thm.~5.6, \S 6.2]{Dold:1957},
  we conclude that \(\GP{C(a_{g})}{\Gamma}\) is again homotopy equivalent to the identity,
  so that \(g\) acts trivially in the (co)homology of~\(\GP{X}{\Gamma}\).

  The converse follows from Lemma\nobreakspace \ref {thm:surjection} as in the proof of Proposition\nobreakspace \ref {thm:converseGP}.
\end{proof}

\begin{proof}[Proof of Proposition\nobreakspace \ref {thm:SPm-weakly-tight}]
  Let \(Y\) be a finite \(C\)-CW~complex;
  then \(\SP{Y}{n}\) is again a finite \(C\)-CW~complex.
  Recall from~\cite[Prop.~3.7]{Baird}
  that the Serre spectral sequence
  for~\(Y_{C}\) degenerates at the second page if and only if
  \begin{equation}
    \label{eq:crit-weakly-tight}
    \dim H^{*}(Y^{\tau}) = \dim H^{*}(Y)\triv.
  \end{equation}
  (In~\cite{Baird}, \(Y\) is assumed to be a compact \(\Z_{2}\)-manifold;
  the proof carries over to our setting.)
  We refer to~\cite[Sec.~3.2]{Baird} for the definition of~\(V\triv\) for a finite-dimensional \(\Z_{2}\)-vector space~\(V\)
  with a linear involution; we only observe that for any two such vector spaces we have isomorphisms
  \begin{equation}
    \label{eq:sum-tensor-triv}
    (V\oplus W)\triv \cong V\triv\oplus W\triv
    \quad\text{and}\quad
    (V\otimes W)\triv \cong V\triv\otimes W\triv.
  \end{equation}
  (By the first isomorphism, it is enough to verify the second for the trivial
  and the \(2\)-dimensional indecomposable \(\Z_{2}\)-module.)
  Moreover, if \(U\) is a finite-dimensional vector space
  and \(\Z_{2}\) acts on ~\(U\otimes U\) by swapping the factors, then
  \begin{equation}
    \label{eq:tensor-swap-triv}
    (U\otimes U)\triv \cong U.
  \end{equation}

  To prove the proposition, assume first that \(X\ne\emptyset\) is connected.
  Then \(H^{0}(X)\triv\) does not vanish, which by~\eqref{eq:crit-weakly-tight}
  implies that \(X\) has fixed points.
  Our claim therefore reduces to Baird's result~\cite[Prop.~3.9]{Baird}.

  Next we recall that if \(X=Y\sqcup Z\) is the disjoint union
  of two subspaces, then
  \begin{equation}
    \label{eq:SP-Y-cup-Z}
    \SP{X}{n} = \!\bigsqcup_{k+l=n}\!\SP{Y}{k}\times\SP{Z}{l},
  \end{equation}
  see~\cite[eq.~(8.8)]{Dold:1957}.

  Assume that \(\tau\) transposes \(Y\) and \(Z\approx Y\), \cf~\cite[Prop.~3.2]{Baird}, so that \(H^{*}(X)\triv=0\).
  The subspaces in~\eqref{eq:SP-Y-cup-Z} are then also permuted by~\(\tau\).
  If \(n\) is odd, none of them is \(\tau\)-stable.
  Hence there are no fixed points and \(H^{*}(\SP{X}{n})\triv=0\), proving our claim.

  If \(n=2k\) is even, then \(\SP{X}{n}^{\tau} = (\SP{Y}{k}\times\SP{Z}{k})^{\tau} \approx \SP{Y}{k}\) and
  \begin{align}
    H^{*}(\SP{X}{n})\triv &= H^{*}(\SP{Y}{k}\times\SP{Z}{k})\triv \\
    &\cong \bigl(H^{*}(\SP{Y}{k})\otimes H^{*}(\SP{Y}{k})\bigr)\triv
    \cong H^{*}(\SP{Y}{k})
  \end{align}
  by~\eqref{eq:tensor-swap-triv}.
  Thus, the criterion~\eqref{eq:crit-weakly-tight} is again satisfied.

  Consider finally the case of general~\(X\).
  Its finitely many connected components are either stable under~\(\tau\)
  or come in pairs that are transposed by~\(\tau\). If both~\(Y\) and~\(Z\) in~\eqref{eq:SP-Y-cup-Z} are \(\tau\)-stable,
  then we also have
  \begin{equation}
    \SP{X}{n}^{\tau} = \!\bigsqcup_{k+l=n}\!\SP{Y}{k}^{\tau}\times\SP{Z}{l}^{\tau}.
  \end{equation}
  Hence the claim follows from the two cases already discussed together with the Künneth formula
  and the identities~\eqref{eq:sum-tensor-triv}.
\end{proof}

\begin{proof}[Proof of Example\nobreakspace \ref {ex:SP2-SigmaT}]
  Recall that \(\kk=\R\) in this example and write \(I=[0,1]\). The projection \(\pi\colon X=(G\times I)/{\sim}\to I\)
  induces a projection~\(\SP{\pi}{2}\)
  from~\(\SP{X}{2}\) onto~\(\SP{I}{2}\), which we identify with the triangle with vertices~\((0,0)\),~\((0,1)\) and~\((1,1)\).
  
  We consider the spectral sequence~\(E_{k}^{p,q}\) induced by the filtration of~\(\SP{X}{2}\)
  by the inverse images of the faces of this triangle and converging to~\(H_{G}^{*}(\SP{X}{2})\).
  The fibre over each vertex is a fixed point and contributes a summand~\(H^{*}(BG)\)
  to~\(E_{1}^{0,*}\), the zeroeth column of the first page of the spectral sequence.
  The fibre over each leg is \(G\) and contributes a copy of~\(\R\) to~\(E_{1}^{1,*}\).
  The fibre over the hypotenuse is \(\SP{G}{2}\), on which \(G\) acts locally freely.
  Hence,
  \begin{equation}
    H_{G}^{*}(\SP{G}{2})=H^{*}(\SP{G}{2}/G)=H^{\mathrm{even}}(G),
  \end{equation}
  as can be seen
  by first dividing \(G\times G\) by the diagonal \(G\)-action and then by~\(S_{2}\).
  The fibre over the interior of the triangle is \(G\times G\), hence contributes \(H^{*}(G)\)
  to~\(E_{1}^{2,*}\).

  For~\(q>0\), the differential~\(d_{1}\colon E_{1}^{1,q}\to E_{1}^{2,q}\)
  is the inclusion~\(H^{\mathrm{even}}(G)\hookrightarrow H^{*}(G)\),
  and the zeroeth row~\(E_{1}^{*,0}\) computes the cohomology of the triangle.
  This implies \(E_{2}^{1,*}=0\) and \(E_{2}^{2,*}=H^{\mathrm{odd}}(G)\).

  We claim that the differential
  \begin{equation*}
    d_{2}^{\mkern1.5mu 0,2s}\colon E_{2}^{0,2s} \to E_{2}^{2,2s-1}
  \end{equation*}
  vanishes for~\(s>1\). Inspired by the proof of~\cite[Thm.~5]{Totaro:2014},
  we consider the squaring map on~\(G\).
  It induces a map~\(X\to X\) preserving the fibres of~\(\pi\), hence also a map
  \(\SP{X}{2}\to\SP{X}{2}\) preserving the fibres of~\(\SP{\pi}{2}\).
  We therefore get a map of spectral sequences which scales \(E_{2}^{0,2s}\) by~\(2^{s}\)
  and \(E_{2}^{2,2s-1}\) by~\(2^{2s-1}\). Because this map commutes with the differentials in the
  spectral sequence and \(2^{s}\ne 2^{2s-1}\) for~\(s>1\), we conclude \(d_{2}^{\mkern1.5mu 0,2s}=0\) for~\(s>1\).
  Hence \(E_{\infty}^{2,*}=E_{3}^{2,*}=H^{\mathrm{odd}}(G)\), except possibly in degree~\(1\).
  Since \(r\ge3\), this shows that \(E_{\infty}^{2,*}\) is non-zero and finite-dimensional over~\(\R\).
  As mentioned earlier, this implies that \(H_{G}^{*}(\SP{X}{2})\) has torsion over~\(H^{*}(BG)\).
\end{proof}

\end{document}